\documentclass[12pt]{article}

\usepackage[margin=1in]{geometry}
\usepackage{amsmath,amssymb,amsthm}
\usepackage[authoryear,round]{natbib}
\usepackage{enumitem}
\usepackage{hyperref}
\usepackage{times}
\hypersetup{colorlinks=true,linkcolor=blue,citecolor=blue,urlcolor=blue}

\newtheorem{theorem}{Theorem}
\newtheorem{proposition}[theorem]{Proposition}

\newcommand{\E}{\mathbb{E}}
\newcommand{\bbP}{\mathbb{P}}

\newcommand{\R}{\mathbb{R}}
\newcommand{\one}{\mathbf{1}}

\title{Minimax unbiased estimation for finite populations with bounded outcomes}
\author{P. M. Aronow (Yale) and Patrick Lopatto (UNC)}
\date{\today}

\begin{document}
\maketitle

\begin{abstract}
We study design-unbiased estimation of the finite-population total $\sum_{i=1}^N y_i$ when each outcome satisfies known bounds \(y_i\in[a_i,b_i]\). For any sampling design with inclusion probabilities \(\pi_i>0\), we prove
a sharp lower bound on the worst-case squared error over the rectangular
parameter space. This bound is attained if and only if the unit inclusion
indicators are pairwise independent, in which case the minimax estimator is
the midpoint-differenced Horvitz--Thompson estimator
$\
\sum_{i=1}^N m_i+\sum_{i\in S}(y_i-m_i)/\pi_i,$ with
$m_i=(a_i+b_i)/{2}$.
We then solve the joint design-and-estimation problem under the constraint
\(\sum_i \pi_i\le n\). We find that a minimax strategy samples units independently
with probabilities $\pi_i^\ast=\min(1,c (b_i-a_i))$
where \(c>0\) is chosen so that \(\sum_i \pi_i^\ast=n\), and uses the
midpoint-differenced estimator. This extends \citet{Gabler1990}'s linear minimax result
to the full class of design-unbiased estimators. We also show that the
estimator is admissible among unbiased estimators and affine equivariant.
\end{abstract}

\section{Introduction}

We consider a finite population $y=(y_1, y_2, \dots , y_N)$ of real numbers and the problem of estimating the total $T(y) = \sum_{i=1}^N y_i$ using a subset $(y_i)_{i \in S}$, where the index set $S\subset \{1, \dots, N\} $ is chosen randomly using a known sampling distribution. 
This problem is classical and has been studied previously under various combinations of assumptions on the population and the sampling design. 

For a broad class of designs, \citet{Godambe1955} showed that, if no additional restrictions are placed on the $y_i$, no single unbiased linear estimator minimizes variance uniformly for all population values.  In a similar setting with unrestricted outcomes, \citet{GodambeJoshi1965} showed that the well-known Horvitz--Thompson estimator \citep{HorvitzThompson1952} is admissible among all unbiased estimators. Under the condition that the population variance is bounded by a known constant, \citet{BickelLehmann1981} showed that the sample mean is minimax with respect to expected squared error under simple random sampling, generalizing earlier work of  \citet{Aggarwal1959}. A variation in which each unit has an associated covariate was considered in \citet{ChengLi1983}. Further, \citet{joshi1979best} proved a minimax-type result for simple random sampling that applies to arbitrary convex loss functions.

In this note, we are interested in the case where each $y_i$ satisfies  $y_i \in [a_i , b_i]$ for some constants $a_i$ and $b_i$ known ex ante. 
When the bounds $a_i$ and $b_i$ are the same for all $i$, the minimax estimator under simple random sampling was obtained by \citet{hodges1982minimax}. When the bounds vary by index, exact minimax results are difficult to obtain. Conditional minimax estimators were investigated by \citet{gabler1988conditional}, and asymptotically minimax estimators were given by \citet{stenger1989asymptotic}. When $a_i = 0$ for all $i$ and the units are sampled independently with probability $\pi_i = n b_i / \sum_{i=1}^N b_i$, and $n$ is a fixed constant such that all $\pi_i \leq 1$, a minimax estimator with respect to squared error loss among all unbiased inhomogeneous linear estimators was derived by \citet{Gabler1990}. Gabler further shows that this strategy is minimax when restricted to unbiased inhomogeneous linear estimators. 

To the best of our knowledge, no minimax estimator among the class of design-unbiased estimators has previously been identified for bounded outcomes without imposing a fixed-size design restriction or restricting attention to linear estimators. We consider the case of arbitrary bounds $a_i$ and $b_i$ that may vary with the index $i$. We provide three results that give a relatively complete solution of this problem. 
The first, Theorem~\ref{thm:lower}, gives a lower bound on the expected squared error of any unbiased estimator. The second, Theorem~\ref{thm:sharpness}, shows that this lower bound is sharp exactly when the events for each unit included in the random sample are pairwise independent. The third, Theorem~\ref{thm:design}, derives a minimax unbiased statistical strategy under an expected sample size constraint. The resulting strategy is essentially equivalent to that of \citet{Gabler1990}'s strategy in the linear case. We further show that the associated estimator, a midpoint-differenced Horvitz-Thompson estimator, is (i) minimax among unbiased estimators when the sampling design is pairwise independent, (ii) admissible in the class of unbiased estimators, and  (iii) affine equivariant.  We state these results formally after introducing some notation.

\section{Setup and notation}

Assume $-\infty<a_i\le b_i<\infty$,  and let
\[
  \Theta=\prod_{i=1}^N [a_i,b_i],\qquad
  m_i=(a_i+b_i)/2,
  \qquad
  r_i=(b_i-a_i)/2.
\]
We suppose $r_i>0$ for all indices $i$ to avoid degenerate cases with known outcomes.

Define a sampling design as a probability distribution $p$ on subsets $s\subseteq \{1, \dots, N\}$.  Let $S$ denote the random sample and let $\pi_i=\bbP_p(i\in S)$, where $\bbP_p$ denotes probability with respect to $p$. We suppose that $\pi_i>0$ for all $i$.  An estimator is a collection of measurable functions $\delta_s \colon \Theta_s \rightarrow \R$, one for each possible sample, where $\Theta_s$ denotes the projection of $\Theta$ onto the indices contained in $s$.  We say an estimator is unbiased if for all $y \in \Theta$, 
\[
  \E_p\big[ \delta_S(y_S)\big]=T(y) , \qquad 
  y_{S} = (y_i)_{i \in S}.
\]
We evaluate estimators using squared error loss, defined as 
\[
  R(\delta,p;y)=\E_p\big [\big(\delta_S(y_S)-T(y)\big)^2\big ].
\]
We emphasize that all estimators we consider are deterministic functions of the sample $y_S$. 
Randomized estimators need not be considered separately, since conditioning on the sample and observed values cannot increase the  risk.

We define a generalized difference estimator 
\[
  \widehat T (y_S)
  =\sum_{i=1}^N m_i+
    \sum_{i\in S}\frac{y_i-m_i}{\pi_i},
\]
which is well-known in the literature \citep{CasselSarndalWretman1976,DevilleSarndal1992}. We refer to $\widehat T (y_S)$ as the midpoint-differenced Horvitz--Thompson estimator, using $\widehat T$ as shorthand.

\section{Main results}

We now state our three main theorems, deferring their proofs to the following section. The first provides a lower bound on the risk of any unbiased estimator.

\begin{theorem}
\label{thm:lower}
For any sampling design $p$ with all $\pi_i>0$ and any unbiased estimator $\delta$,
\[
  \sup_{y\in\Theta}R(\delta,p;y)
  \geq
  \sum_{i=1}^N r_i^2\left( \frac{1-\pi_i}{\pi_i}\right).
\]
\end{theorem}

For a general design, write
\[
  I_i=\one(i\in S),\quad
  \pi_{ij}=\bbP_p(I_i=1,I_j=1),\quad
  \Delta_{ij}=\pi_{ij}-\pi_i\pi_j, \quad   D_\pi=\sum_{i=1}^N r_i^2\left(\frac{1-\pi_i}{\pi_i}\right).
\]
The inclusion indicators are pairwise independent if $\Delta_{ij}=0$ for all $i\ne j$. Our second theorem shows that the lower bound  from the previous theorem is sharp if and only if this condition holds.

\begin{theorem}
\label{thm:sharpness}
Assume $r_i>0$ for all $i$.  Fix a design $p$ with $\pi_i>0$ for all $i$.
The following statements are equivalent:
\begin{enumerate}[label=\textup{(\roman*)}]
\item there exists an unbiased estimator $\delta$ with
\[
  \sup_{y\in\Theta}R(\delta,p;y)=D_\pi;
\]
\item the midpoint-differenced estimator $\widehat T$ satisfies
\[
  \sup_{y\in\Theta}R(\widehat T,p;y)=D_\pi;
\]
\item $\Delta_{ij}=0$ for every $i\ne j$.
\end{enumerate}
\end{theorem}

Our final main theorem concerns the simultaneous choice of both the design and the estimator. We generalize a result of \citet{Gabler1990}, who showed an analogous result for the family of unbiased inhomogeneous linear estimators. We consider the optimal strategy under the constraint that the expected sample size satisfies  $\E_p [|S|]=\sum_{i=1}^N\pi_i\leq n$, where $0<n< N$ is a parameter. (The case $n=N$ is trivial.) Define
\[
  V_n =
  \min_{0<\pi_i\leq 1,\,\sum_{i=1}^N\pi_i\leq n}
  \sum_{i=1}^Nr_i^2 \left( \frac{1-\pi_i}{\pi_i}\right).
\]

\begin{theorem}
\label{thm:design}
Among all pairs $(p,\delta)$ in which $p$ is a sampling design such that $\pi_i >0$ for all $i$ and  $\E_p[|S|] \leq n$, and $\delta$  is unbiased,
\[
  \inf_{p,\delta}\sup_{y\in\Theta}R(\delta,p;y)=V_n.
\]
The value $V_n $ is attained by using the estimator $\widehat T$ and sampling each unit independently with inclusion probabilities 
\[
  \pi_i^*=\min(1,cr_i),
\]
where the constant $c>0$ is chosen so that $\sum_{i=1}^N\pi_i^*=n$.
\end{theorem}

We also include two supplementary propositions. The first shows that $\widehat T$ is admissible in the class of unbiased estimators, even when pairwise independence does not hold.

\begin{proposition}\label{p:admissible}
Fix a sampling design $p$ with $\pi_i>0$ for every $i$. The midpoint-differenced estimator $\widehat T$ 
is admissible in the class of unbiased estimators under squared-error loss.
\end{proposition}

The second shows that $\widehat T$ is equivariant under affine transformations. 

\begin{proposition}\label{p:equivariant}
Fix a sampling design $p$ with inclusion probabilities $\pi_i>0$. Let
$\lambda>0$ and $c_1,\ldots,c_N\in\mathbb R$. Define the transformed
population values
\[
y_i^{\lambda,c}=\lambda y_i+c_i
\]
and the transformed parameter space
\[
\Theta^{\lambda,c}
=
\prod_{i=1}^N[\lambda a_i+c_i,\lambda b_i+c_i].
\]
Let $\widehat T^{\lambda,c}$ denote the difference estimator in the transformed
problem, computed using the same sampling design $p$. Then, for every
$y\in\Theta$ and every sample $s$,
\[
\widehat T^{\lambda,c}_s(y_s^{\lambda,c})
=
\lambda \widehat T_s(y_s)+\sum_{i=1}^N c_i.
\]
\end{proposition}

\section{Discussion}

This paper shows that known bounds on finite-population outcomes alter the classical design-unbiased estimation problem in a sharp and constructive way. The bounds identify a known midpoint total that can be used in a differenced Horvitz--Thompson estimator. We show that this is exactly minimax among all design-unbiased estimators, including nonlinear estimators, under pairwise-independent sampling. The proof clarifies why this is so. On the vertices of the parameter rectangle, unbiasedness fixes the first-order components needed to estimate the total, while higher-order components cannot reduce the required second moment. Thus the worst-case risk lower bound is unavoidable, and the difference estimator attains it precisely when second-order inclusion covariances vanish.

The design result gives a corresponding minimax strategy when the sampling design is also chosen subject to an expected sample size constraint. The optimal inclusion probabilities are proportional to the outcome radii, with truncation at one, and independent sampling with these probabilities attains the minimax value. This characterizes how sampling effort should be allocated when units differ in the width of their known outcome bounds: units with larger possible deviations from their midpoint receive higher inclusion probabilities. Together, the lower bound, sharpness characterization, and optimal-design theorem provide an exact finite-sample theory for minimax design-unbiased estimation over rectangular parameter spaces. They also separate the setting treated here from fixed-size designs, where dependence among inclusion indicators generally prevents attainment of the same universal lower bound.

\section{Proofs}

We begin with the proof of our first theorem.

\begin{proof}[Proof of Theorem~\ref{thm:lower}]
It suffices to lower-bound $R$ over the vertices of $\Theta$.
Write
\[
  y_i=m_i+r_i\varepsilon_i,
  \qquad \varepsilon_i\in\{-1,1\},
\]
and let 
$\varepsilon=(\varepsilon_1,\ldots,\varepsilon_N)$ be uniformly distributed independently of the
sampling design.  Recall that  $S$ denotes the random sample and we defined the indicator functions 
$I_i=\mathbf 1_{\{i\in S\}}$.  Define
\[
  M=\sum_{i=1}^Nm_i,
  \qquad
  F(\varepsilon)=T(y)-M=\sum_{i=1}^Nr_i\varepsilon_i.
\]
For  each sample $s$, write
\[
  \varepsilon_s=(\varepsilon_i)_{i\in s},
  \quad
  m_s=(m_i)_{i\in s},
  \quad
  r_s=(r_i)_{i\in s},
\quad
  m_s+r_s\varepsilon_s
  =
  (m_i+r_i\varepsilon_i)_{i\in s}.
\]
 Set
\[
  D_s(\varepsilon_s)=\delta_s(m_s+r_s\varepsilon_s)-M, \qquad 
  D=D_S(\varepsilon_S).
\]

Let $\bar R(\delta,p)$ be the Bayes risk under the uniform prior on the
vertices.  By assumption, $\bar R(\delta,p)<\infty$, and $F$ and $D$ have finite second moments.

Unbiasedness implies that, for every realization of the full sign vector $\varepsilon$,
\begin{equation}\label{e:1}
  \E_S[D\mid \varepsilon]=F(\varepsilon).  
\end{equation}
Therefore
\[
  \bar R(\delta,p)
  =\E_{\varepsilon,S}\big[ D-F(\varepsilon)\big]^2  =\E_{\varepsilon,S}[D^2]-\E_\varepsilon \big[ F(\varepsilon)^2\big],
\]
because
\[
  \E_{\varepsilon,S}\big [D F(\varepsilon)\big ]
  =
  \E_\varepsilon\!\big[
    F(\varepsilon)\E_S[D\mid\varepsilon]
  \big]
  =
  \E_\varepsilon \big[ F(\varepsilon)^2 \big].
\]

We now lower-bound $\E [ D^2]$ by projecting $D$ onto the span of the  observed
sign variables.  Work in the Hilbert space
$L^2(\mathbb P_p\otimes \mathrm{Unif}\{-1,1\}^N)$ and define $
  U_i=I_i\varepsilon_i$. 
For $i\neq j$,
\[
  \langle U_i,U_j\rangle
  =
  \E[I_iI_j\varepsilon_i\varepsilon_j]
  =
  \E[I_iI_j]\E[\varepsilon_i\varepsilon_j]
  =
  0,
\]
so the variables $U_i$ are mutually orthogonal.  Also
\[
  \|U_i\|_2^2=\E [I_i]=\pi_i .
\]

We next compute the correlations of $D$ with these directions.  By \eqref{e:1},
\[
  \E[D\varepsilon_i]
  =
  \E\big [\varepsilon_i\E_S[D\mid\varepsilon]\big]
  =
  \E\big[\varepsilon_iF(\varepsilon)\big]
  =
  r_i .
\]
Moreover,
\[
\begin{aligned}
  \E[D(1-I_i)\varepsilon_i]
  &=
  \sum_{s\not\ni i}p(s)\,
  \E_\varepsilon\big[
    D_s(\varepsilon_s)\varepsilon_i
  \big]
  =
  0.
\end{aligned}
\]
Indeed, if $s\not\ni i$, then $D_s(\varepsilon_s)$ depends only on the
observed signs $\{\varepsilon_j:j\in s\}$, which are independent of
$\varepsilon_i$, and $\E[\varepsilon_i]=0$.  Hence
\[
  \langle D,U_i\rangle
  =
  \E[D I_i\varepsilon_i]
  =
  \E[D\varepsilon_i]
   -\E\big[D(1-I_i)\varepsilon_i\big]
  =
  r_i .
\]

By Bessel's inequality applied to the orthogonal family
$\{U_i\}_{i=1}^N$,
\[
  \E [D^2]
  =
  \|D\|_2^2
  \geq
  \sum_{i=1}^N
  \frac{\langle D,U_i\rangle^2}{\|U_i\|_2^2}
  =
  \sum_{i=1}^N \frac{r_i^2}{\pi_i}.
\]
Finally,
\[
  \E \big[ F(\varepsilon)^2 \big]
  =
  \E\left(\sum_{i=1}^Nr_i\varepsilon_i\right)^2
  =
  \sum_{i=1}^Nr_i^2.
\]
Therefore
\[
  \bar R(\delta,p)
  \geq
    \sum_{i=1}^N \frac{r_i^2}{\pi_i}-  \sum_{i=1}^Nr_i^2
  =
  \sum_{i=1}^N  r_i^2 \left( \frac{1-\pi_i}{\pi_i} \right).
\]
Since the supremum risk over $\Theta$ is at least the average risk over the
vertices, the claimed lower bound follows.
\end{proof}

Before proceeding to the proof of the second theorem, we require the following proposition.

\begin{proposition}
\label{thm:pairwise}
Suppose that $\Delta_{ij}=0$ for every pair $i\ne j$. Then
\[
  \inf_{\delta}
  \sup_{y\in\Theta}R(\delta,p;y)
  =
  \sum_{i=1}^N r_i^2 \left( \frac{1-\pi_i}{\pi_i}\right),
\]
where the infimum is over all unbiased estimators, and the infimum is attained by $\widehat T$.
\end{proposition}

\begin{proof}
 The error of $\widehat T$ is
\[
  \widehat T-T(y)
  =\sum_{i=1}^N\left(\frac{I_i}{\pi_i}-1\right)z_i, \qquad z_i=y_i-m_i.
\]
The previous display implies that for an arbitrary design,
\[
\begin{aligned}
  R(\widehat T,p;y)
  &=\sum_{i=1}^N\frac{1-\pi_i}{\pi_i}z_i^2
    +2\sum_{i<j}\frac{\Delta_{ij}}{\pi_i\pi_j}z_i z_j .
\end{aligned}
\]
Under the pairwise-independence assumption, this reduces to
\[
  R(\widehat T,p;y)
  =\sum_{i=1}^N\frac{1-\pi_i}{\pi_i}(y_i-m_i)^2 .
\]
Maximizing over $\Theta$ gives
\[
  \sup_{y\in\Theta}R(\widehat T,p;y)
  =\sum_{i=1}^Nr_i^2\frac{1-\pi_i}{\pi_i},
\]
which matches Theorem~\ref{thm:lower}.
\end{proof}

\begin{proof}[Proof of Theorem~\ref{thm:sharpness}]
Statement \textup{(iii)} implies \textup{(ii)} by
Proposition~\ref{thm:pairwise}, and \textup{(ii)} implies \textup{(i)}
because $\widehat T$ is unbiased.  It remains to prove that
\textup{(i)} implies \textup{(iii)}.

Suppose that \textup{(i)} holds.  Use the notation and the uniform vertex
prior from the proof of Theorem~\ref{thm:lower}, and recall that $\bar R(\delta,p)$ denotes the Bayes risk under the uniform prior on the
vertices.  Since $\sup_{y\in\Theta}R(\delta,p;y)=D_\pi$, we have $
  \bar R(\delta,p)\leq D_\pi$. 
By the vertex-prior argument in the proof of Theorem~\ref{thm:lower}, $\bar R(\delta,p)\geq D_\pi$.  Hence
\begin{equation}\label{e:2}
  \bar R(\delta,p)=D_\pi .                   
\end{equation}

In the proof of Theorem~\ref{thm:lower}, we showed
\[
  \bar R(\delta,p)
  =
  \|D\|_2^2-\|F\|_2^2
  =
  \|D\|_2^2-\sum_{i=1}^Nr_i^2,
\]
and also
\[
  \langle D,U_i\rangle=r_i,
  \qquad
  \|U_i\|_2^2=\pi_i,
  \qquad
  \langle U_i,U_j\rangle=0 \quad (i\ne j).
\]
Let $V=\operatorname{span}\{U_1,\ldots,U_N\}$.  Since the $U_i$ are
orthogonal, the orthogonal projection of $D$ onto $V$ is
\[
  P_VD
  =
  \sum_{i=1}^N
  \frac{\langle D,U_i\rangle}{\|U_i\|_2^2}U_i
  =
  \sum_{i=1}^N\frac{r_i}{\pi_i}U_i .
\]
Therefore
\[
  \|P_VD\|_2^2
  =
  \sum_{i=1}^N\frac{r_i^2}{\pi_i}.
\]
On the other hand, by \eqref{e:2} and the definition of $D_\pi$,
\[
  \|D\|_2^2
  =
  \bar R(\delta,p)+\sum_{i=1}^Nr_i^2
  =
  \sum_{i=1}^Nr_i^2\left(\frac{1-\pi_i}{\pi_i}\right)+\sum_{i=1}^Nr_i^2
  =
  \sum_{i=1}^N\frac{r_i^2}{\pi_i}.
\]
Thus
\[
  \|D\|_2^2=\|P_VD\|_2^2.
\]
Since $P_VD$ is the orthogonal projection of $D$ onto $V$, this forces $D=P_VD$. Hence, for every sample $s$ with $p(s)>0$ and every sign vector
$\varepsilon$,
\[
  D_s(\varepsilon_s)
  =
  \sum_{i\in s}\frac{r_i}{\pi_i}\varepsilon_i .
\]
Equivalently, at every vertex $y=m+r\varepsilon$ and every sample $s$ with
$p(s)>0$,
\[
  \delta_s(y_s)
  =
  \sum_{i=1}^N m_i +\sum_{i\in s}\frac{r_i}{\pi_i}\varepsilon_i
  =
  \sum_{i=1}^N m_i +\sum_{i\in s}\frac{y_i-m_i}{\pi_i}
  =
  \widehat T(y_s).
\]
Therefore
\[
  R(\delta,p;m+r\varepsilon)
  =
  R(\widehat T,p;m+r\varepsilon)
\]
for every vertex $m+r\varepsilon$.

Since for each vertex, the  risk of $\delta$ is at most $D_\pi$, and the average risk over
the vertices is $\bar R(\delta,p)=D_\pi$, each vertex risk must in fact
equal $D_\pi$.  Hence
\[
  R(\widehat T,p;m+r\varepsilon)=D_\pi
\]
for every $\varepsilon\in\{-1,1\}^N$.

For the midpoint-differenced estimator,
\[
  \widehat T-T(y)
  =
  \sum_{i=1}^N\left(\frac{I_i}{\pi_i}-1\right)(y_i-m_i).
\]
Thus, at the vertex $y=m+r\varepsilon$,
\[
  R(\widehat T,p;m+r\varepsilon)
  =
  D_\pi
  +
  2\sum_{i<j}
  \frac{\Delta_{ij}}{\pi_i\pi_j}r_ir_j
  \varepsilon_i\varepsilon_j .
\]
Since this equals $D_\pi$ for every sign vector $\varepsilon$,
\[
  \sum_{i<j}
  \frac{\Delta_{ij}}{\pi_i\pi_j}r_ir_j
  \varepsilon_i\varepsilon_j
  =
  0
  \qquad
  \text{for every }\varepsilon\in\{-1,1\}^N.
\]
By orthogonality of the Walsh functions
$\{\varepsilon_i\varepsilon_j:i<j\}$, every coefficient in this expansion
is zero.  Therefore
\[
  \frac{\Delta_{ij}}{\pi_i\pi_j}r_ir_j=0
  \qquad
  \text{for every }i<j.
\]
Because $r_i>0$ and $\pi_i>0$ for all $i$, it follows that
$\Delta_{ij}=0$ for every $i\ne j$.  This proves \textup{(iii)}.
\end{proof}

Finally, we give the proof of our third main theorem.

\begin{proof}[Proof of Theorem~\ref{thm:design}]
By Theorem~\ref{thm:lower}, every admissible design--estimator pair satisfies
\[
  \sup_{y\in\Theta}R(\delta,p;y)
  \geq
  D_\pi.
\]
Further, every  design satisfying the constraint $\E_p[|S|] \leq n$ yields a vector $\pi$ of inclusion probabilities satisfying
\begin{equation}\label{e:domain}
  0<\pi_i\leq 1,
  \qquad
  \sum_{i=1}^N\pi_i\leq n.
\end{equation}
Consequently, to obtain a universal lower bound, it suffices to minimize
$D_\pi$ over this set of vectors.  Since
\[
  D_\pi
  =
  \sum_{i=1}^N\frac{r_i^2}{\pi_i}-\sum_{i=1}^N r_i^2,
\]
this is equivalent to minimizing
\[
  f(\pi)=\sum_{i=1}^N\frac{r_i^2}{\pi_i}.
\]

Let $c>0$ be the unique solution of
\begin{equation}
  \sum_{i=1}^N \min(1,cr_i)=n.
  \label{eq:design-waterfill-equation}
\end{equation}
Indeed, the function
\[
  H(c)=\sum_{i=1}^N\min(1,cr_i)
\]
is continuous, satisfies $H(0)=0$ and $\lim_{c\to\infty}H(c)=N$, and is
strictly increasing on the set where $H(c)<N$.  Since $0<n<N$,
\eqref{eq:design-waterfill-equation} has a unique solution.  Define
\[
  \pi_i^*=\min(1,cr_i),
  \qquad
  \lambda=c^{-2}.
\]
Then $\pi^*$ satisfies \eqref{e:domain}.

We claim that $\pi^*$ minimizes $f$.  Let $\pi$ be any vector satisfying \eqref{e:domain}.  We first prove, for each $i$, that
\begin{equation}
  \frac{r_i^2}{\pi_i}
  -
  \frac{r_i^2}{\pi_i^*}
  \geq
  -\lambda(\pi_i-\pi_i^*).
  \label{eq:design-elementary-support}
\end{equation}
We use the elementary identity
\begin{equation}\label{e:elementary}
  \frac{r_i^2}{x}
  -
  \frac{r_i^2}{a}
  +
  \frac{r_i^2}{a^2}(x-a)
  =
  \frac{r_i^2(x-a)^2}{a^2x}
  \geq 0,
  \qquad x>0,\ a>0.
\end{equation}
If $\pi_i^*<1$, then $\pi_i^*=cr_i$, and hence
\[
  \frac{r_i^2}{(\pi_i^*)^2}
  =
  c^{-2}
  =
  \lambda.
\]
Applying \eqref{e:elementary} with $x=\pi_i$ and $a=\pi_i^*$ gives
\eqref{eq:design-elementary-support}.

If $\pi_i^*=1$, then $cr_i\geq1$, so $r_i^2\geq c^{-2}=\lambda$.
Applying \eqref{e:elementary} with $x=\pi_i$ and $a=1$ gives
\[
  \frac{r_i^2}{\pi_i}-r_i^2
  \geq
  -r_i^2(\pi_i-1).
\]
Since $\pi_i\leq1$ and $r_i^2\geq\lambda$,
\[
  -r_i^2(\pi_i-1)
  \geq
  -\lambda(\pi_i-1).
\]
Thus \eqref{eq:design-elementary-support} also holds in the case $\pi_i^*=1$.

Summing \eqref{eq:design-elementary-support} over $i$ yields
\[
  f(\pi)-f(\pi^*)
  \geq
  -\lambda\sum_{i=1}^N(\pi_i-\pi_i^*)
  =
  \lambda\left(n-\sum_{i=1}^N\pi_i\right)
  \geq 0.
\]
Therefore $\pi^*$ minimizes $f$, and hence minimizes $D_\pi$.  Moreover,
$f$ is strictly convex, since each function
$x\mapsto r_i^2/x$ is strictly convex for $r_i>0$.  The feasible set is
convex, so this minimizer is unique.

By the definition of $V_n$, $
  D_{\pi^*}=V_n$. 
Consequently, every admissible design--estimator pair satisfies
\[
  \sup_{y\in\Theta}R(\delta,p;y)
  \geq
  D_\pi
  \geq
  D_{\pi^*}
  =
  V_n.
\]

It remains to show that independent
sampling with inclusion probabilities $\pi_i^*$ attains equality.  We have 
\[
  \E_p \big[ |S| \big]=\sum_{i=1}^N\pi_i^*=n,
\]
and the inclusion indicators are pairwise independent, so
\[
  \Delta_{ij}=0
  \qquad
  \text{for all }i\ne j.
\]
For this design, the midpoint-differenced estimator $\widehat T$ 
is unbiased, and Proposition~\ref{thm:pairwise} gives
\[
  \sup_{y\in\Theta}R(\widehat T ,p;y)
  =
  D_{\pi^*}
  =
  V_n.
\]
This concludes the proof.
\end{proof}

We now turn to the proofs of our two supplementary propositions. 

\begin{proof}[Proof of Proposition~\ref{p:admissible}]
It suffices to rule out domination by another unbiased estimator. Suppose, toward a contradiction, that there exists an unbiased estimator $\delta$ such that
\[
R(\delta,p;y)\le R(\widehat T,p;y)
\qquad\text{for all }y\in\Theta,
\]
with strict inequality at some point $y^0\in\Theta$.

Choose a product probability measure $
\mu=\mu_1\otimes\cdots\otimes\mu_N$ 
on $\Theta$ such that each $\mu_i$ has mean $m_i$, has positive variance
\[
\sigma_i^2=\int (y_i-m_i)^2\,d\mu_i(y_i)>0,
\]
and gives positive mass to $y_i^0$. This is possible since $m_i$ lies in the interior of $[a_i,b_i]$. Then $\mu(\{y^0\})>0$.

Let $Y\sim\mu$, independently of the sample $S\sim p$, and write
\[
Z_i=Y_i-m_i,
\qquad
F=T(Y)-\sum_{i=1}^N m_i =\sum_{i=1}^N Z_i,
\qquad
D=\delta_S(Y_S)-M .
\]
Since $\delta$ is unbiased, $
\E_p[D\mid Y]=F$. 
Further, by the definitions of $D$ and $F$,
\[
D-F=\delta_S(Y_S)-T(Y).
\]
Therefore, conditioning on $Y=y$ and then integrating over $\mu$,
\[
\int_\Theta R(\delta,p;y)\,d\mu(y)
=
\E[(D-F)^2].
\]

We now prove a lower bound for this quantity. Recall that we set $ 
I_i=\mathbf 1_{\{i\in S\}}$, and define $
U_i=I_iZ_i$. 
The random variables $U_i$ are pairwise orthogonal in $L^2$, because the coordinates of $Y$ are independent and centered. Moreover
\[
\|U_i\|_2^2=\E[I_iZ_i^2]=\pi_i\sigma_i^2 .
\]
Also,
\[
\begin{aligned}
\langle D,U_i\rangle
&=\E\!\left[D I_i Z_i\right]  \\
&=\E\!\left[
Z_i\sum_{s\ni i}p(s)\bigl(\delta_s(Y_s)-M\bigr)
\right].
\end{aligned}
\]
The terms with $s\not\ni i$ may be inserted into the last sum, since for such $s$ the quantity
$\delta_s(Y_s)-M$ is independent of $Z_i$ and $\E [Z_i]=0$. Hence
\[
\langle D,U_i\rangle
=
\E\!\left[
Z_i\sum_s p(s)\big(\delta_s(Y_s)-M\big)
\right]  
=\E[Z_iF]
=\sigma_i^2 .
\]
By Bessel's inequality,
\[
\E[D^2]
\ge
\sum_{i=1}^N
\frac{\langle D,U_i\rangle^2}{\|U_i\|_2^2}
=
\sum_{i=1}^N\frac{\sigma_i^2}{\pi_i}.
\]
Since $\E[D F]=\E[F\,\E_p(D\mid Y)]=\E[F^2]$ and
\[
\E[F^2]=\sum_{i=1}^N\sigma_i^2,
\]
we obtain
\[
\int_\Theta R(\delta,p;y)\,d\mu(y)
=
\E[(D-F)^2]
\ge
\sum_{i=1}^N\sigma_i^2\left(\frac1{\pi_i}-1\right).
\]

The midpoint-differenced estimator $\widehat T$ attains this bound. Indeed,
\[
\widehat T- \sum_{i=1}^N m_i
=
\sum_{i=1}^N \frac{I_iZ_i}{\pi_i},
\]
so
\[
\widehat T-T(Y)
=
\sum_{i=1}^N\left(\frac{I_i}{\pi_i}-1\right)Z_i .
\]
Again the cross terms vanish under the centered product prior, and therefore
\[
\int_\Theta R(\widehat T,p;y)\,d\mu(y)
=
\sum_{i=1}^N
\E\!\left[\left(\frac{I_i}{\pi_i}-1\right)^2\right]\sigma_i^2
=
\sum_{i=1}^N\sigma_i^2\left(\frac1{\pi_i}-1\right).
\]
Thus $\widehat T$ minimizes $\mu$-Bayes risk among unbiased estimators.

But $\delta$ was assumed to weakly dominate $\widehat T$ everywhere and to improve it strictly at $y^0$. Since $\mu(\{y^0\})>0$, this would imply
\[
\int_\Theta R(\delta,p;y)\,d\mu(y)
<
\int_\Theta R(\widehat T,p;y)\,d\mu(y),
\]
contradicting the preceding Bayes-risk minimality. Therefore no unbiased estimator dominates $\widehat T$, and $\widehat T$ is admissible in the class of unbiased estimators.
\end{proof}

\begin{proof}[Proof of Proposition~\ref{p:equivariant}]
Recall that $m_i = (a_i + b_i)/2$. In the transformed problem, the
corresponding midpoint is
\[
m_i^{\lambda,c}
=
\frac{(\lambda a_i+c_i)+(\lambda b_i+c_i)}2
=
\lambda m_i+c_i.
\]
Therefore, for any sample $s$,
\[
\begin{aligned}
\widehat T^{\lambda,c}_s(y_s^{\lambda,c})
&=
\sum_{i=1}^N m_i^{\lambda,c}
+
\sum_{i\in s}
\frac{y_i^{\lambda,c}-m_i^{\lambda,c}}{\pi_i} \\
&=
\sum_{i=1}^N(\lambda m_i+c_i)
+
\sum_{i\in s}
\frac{(\lambda y_i+c_i)-(\lambda m_i+c_i)}{\pi_i} \\
&=
\lambda\sum_{i=1}^N m_i+\sum_{i=1}^N c_i
+
\lambda\sum_{i\in s}\frac{y_i-m_i}{\pi_i} \\
&=
\lambda\widehat T_s(y_s)+\sum_{i=1}^N c_i.
\end{aligned}
\]
This is precisely the affine transformation rule induced on the total, since
\[
T(y^{\lambda,c})
=
\sum_{i=1}^N(\lambda y_i+c_i)
=
\lambda T(y)+\sum_{i=1}^N c_i.
\]
Hence the differenced estimator transforms in the same way as the estimand, which
is the claimed equivariance.
\end{proof}

\section*{Declaration of generative AI and AI-assisted technologies}
During preparation of this manuscript, the authors used ChatGPT (GPT 5.5-Pro) to assist
with drafting, editing, and checking arguments. The authors reviewed and edited the content as necessary and take full responsibility for the content of the publication.

\section*{Acknowledgements}
The authors thank Haoge Chang and Ben Green for helpful conversations.

\bibliographystyle{plainnat}
\bibliography{references}

\end{document}